\documentclass[11pt,english]{smfart}

\usepackage{amsmath,leftidx,amsthm,sfheaders}
\usepackage[all]{xy}
\usepackage{xspace,smfthm}
\usepackage[psamsfonts]{amssymb}
\usepackage[latin1]{inputenc}
\usepackage{graphicx,color}
\usepackage{hyperref,fancyhdr}

\newtheorem*{remark}{\bf Remark}

\newtheorem{theorem}{\bf Theorem}[section]
\newtheorem{proposition}[theorem]{\bf Proposition}

\newtheorem*{definition*}{\bf Definition}

\newtheorem*{claim}{\bf Claim}
\newtheorem{lemma}[theorem]{\bf Lemma}

\def\C{{\mathbb C}}
\def\N{{\mathbb N}}

\def\P{\mathbb{P}}

\def\supp{\textup{supp}}

\def\dist{\textup{dist}}

\def\top{\textup{top}}

\def\and{{\quad\text{and}\quad}}

\title{Entropy of meromorphic maps acting on analytic sets}
\author{Henry De Thélin}
\address{LAGA, UMR 7539, Institut Galilée, Université Paris 13, 99 avenue J.B. Clément, 93430 Villetaneuse, France}
\email{dethelin@math.univ-paris13.fr}
\author{Gabriel Vigny}
\address{LAMFA UMR 7352, Universit\'e de Picardie Jules Verne, 33 rue Saint-Leu, 80039 AMIENS Cedex 1, FRANCE}
\email{gabriel.vigny@u-picardie.fr}
\thanks{The second author's research is partially supported by the ANR grant Fatou ANR-17-CE40-0002-01.}

\begin{document}
\begin{abstract} Let $f : X\dasharrow X$ be a dominating meromorphic map on a compact Kähler manifold $X$ of dimension $k$. We extend the notion of topological entropy $h^l_{\mathrm{top}}(f)$ for the action of $f$ on (local) analytic sets of dimension $0\leq l \leq k$. For an ergodic probability measure $\nu$, we extend similarly the notion of measure-theoretic entropy $h_{\nu}^l(f)$.
	
	Under mild hypothesis, we compute $h^l_{\mathrm{top}}(f)$ in term of the dynamical degrees of $f$. In the particular case of endomorphisms of $\P^2$ of degree $d$, we show that $h^1_{\mathrm{top}}(f)= \log d$ for a large class of maps but we give examples where $h^1_{\mathrm{top}}(f)\neq \log d$.
\end{abstract}

\maketitle

\medskip

\noindent \textbf{Mathematics~Subject~Classification~(2010):}
37B40, 37F10, 32Bxx.

\section{Introduction}
Consider a dynamical system $f : X\dasharrow X$ where $f$ is a dominating meromorphic map on a compact Kähler manifold $X$ of dimension $k$ endowed with a Kähler form $\omega$. A central question in the study of such  dynamical system is to compute the \emph{topological entropy} $h_{\top}(f)$ of $f$ and to construct a measure of maximal entropy.

The quantity  $h_{\top}(f)$ is related to the so-called \emph{dynamical degrees} $(d_l(f))_{0\leq l \leq k}$ of $f$. They are defined by (\cite{RS,DS5})
\[ d_l(f) = \lim_{n\to \infty} \left(  \int_X (f^n)^*(\omega^l) \wedge \omega^{k-l} \right)^\frac{1}{n} \]
and the $l$-th degree $d_l(f)$ measures the spectral radius of the action of the pull-back operator $f^*$ on the cohomology group $H^{l,l}(X)$.  It can be shown 
that the sequence of degrees is increasing up to a rank $l$ and then it is decreasing (see \cite{Gromov}). By \cite{gromov_entropy, DS5, DS9}, we always have  $h_{\top}(f) \leq \max_{0 \leq s \leq k} \log d_s$. In order to prove the reverse inequality, the strategy is to construct a measure of maximal entropy $\max_{0 \leq s \leq k} \log d_s$. This has been done in numerous cases (e.g. Hénon maps \cite{BedSmi} or holomorphic endomorphism \cite{ForSib}) and we gave in \cite{ThelinVigny1} a very general criterion under which we can construct a measure $\mu$ of measure-theoretic entropy $h_\mu(f)= \max_{0 \leq s \leq k} \log d_s$. \\

On the other hand, $f$ naturally acts on analytic sets of dimension $l\leq k$ (at least outside the indeterminacy sets), the case $l=0$ being the classical action on points $z\in X$. The purpose of this article is to define natural notions of topological entropy $h^l_{\mathrm{top}}(f)$ and measure-theoretic entropy $h_{\nu}^l(f)$ (for an ergodic invariant probability measure $\nu$) that extend the classical ones and then to compute those entropies. Though such computations will again be in terms of dynamical degrees (see Theorems~\ref{theorem_majoration} and \ref{theorem_minoration}, we shall show in the particular case of endomorphisms of $\P^2$ of degree $d$ that $h^1_{\mathrm{top}}(f)= \log d (=\log d_1)$ for a large class of maps but we give examples where $h^1_{\mathrm{top}}(f) <\log d$ (see Theorem~\ref{generic}). This makes the entropies of meromorphic maps acting on analytic sets richer than the classical notion. \\

Observe for that, in the general setting of compact Kähler manifold, that they are a priori no global analytic sets of positive dimension. This is why, as we will see right below, the point of view we adopt here to define the entropies is the growth rate in a very strong sense of \emph{local analytic sets}. Denote by $I$ the indeterminacy set of $f$. For $\delta >0$ and $n \in \N$, we define:
\begin{align*}
X_l^{\delta,n}:= \Big\{    W \subset X |  &\exists x \in W, \ W \cap  B(x,e^{-n\delta})\ \mathrm{is\ analytic \  of \ exact \ dimension} \ l \ \mathrm{ in } \  B(x,e^{-n\delta}),\\
&\, \mathrm{Vol}_l(W\cap  B(x,e^{-n\delta}))\leq 1, \, \forall k\leq n-1, \  f^k( W) \subset X\backslash I    \Big\} .
\end{align*}
In the above, $\mathrm{Vol}_l$ denotes the $l$-dimensional volume. 
 
For $A,B \subset X$, we denote $\dist(A,B):= \inf \{d(x,y), \ x \in A, \ y \in B\}$. Beware that it is not an actual distance, for example, when $X=\P^k$ and $l=k-1$, two analytic hypersurfaces $A$ and $B$ in $\P^k$ necessarily intersect by Bézout's theorem.
\begin{definition*}
For $n \in \N $, we say that a set $E \subset X_l^{\delta,n}$ is  \emph{$(n,\delta)$-separated} if for all $W\neq W' \in E$ there exists $i \in \{0, \dots, n-1\}$, such that $\dist(f^i(W),f^i(W'))\geq \delta$. 
\end{definition*}

\begin{definition*}
For $l \leq k$, we define $h_{\mathrm{top}}^l(f)$, the \emph{ $l$-topological entropy of $f$}, as the quantity:
\[ h_{\mathrm{top}}^l(f) := \varlimsup_{\delta \to 0} \varlimsup_{n\to \infty}  \frac{1}{n} \log \left(\max  \{\# E, \ E\subset X_l^{\delta,n} \ \mathrm{is} \ (n,\delta)-\mathrm{separated}  \} \right).
 \]
\end{definition*}
\begin{remark}\normalfont
	\begin{enumerate}
		\item When $l=0$,  $h_{\mathrm{top}}^0(f)=h_{\mathrm{top}}(f)$ is the classical topological entropy of $f$ since points whose forward orbit stays in $X\backslash I$ belong to $X_0^{\delta,n}$ for all $n$. By compacity, $h_{\mathrm{top}}^k(f)=0$ and, using a slicing argument, one has that $l \mapsto h_{\mathrm{top}}^l(f)$ is decreasing.
		\item We could also have defined a notion of entropy using $n,\delta$-separated sets in $X_l^{\delta',n}$ for $\delta\neq \delta'$ and then make $\delta\to 0$ and $\delta'\to 0$ in an appropriate order; in here, we choose to take $\delta=\delta'$ in order to simplify the definitions.   
		 \item Finally, observe that our definitions (see also the notion of measure-theoretic entropy below) make sense for a $C^r$-map $f$ on a real $C^r$-manifold $M$ acting on local $C^r$ manifolds.
	\end{enumerate}
\end{remark} 
Our first result is:
\begin{theorem}\label{theorem_majoration}
	Let $f$ be a dominating meromorphic map of a compact Kähler manifold $X$, then for any $0\leq l \leq k$, we have \[h_{\top}^l(f)\leq \log \max_{j\leq k-l }d_j.\]
\end{theorem}
 We now define a measure-theoretic entropy for analytic sets associated to an (ergodic) invariant measure $\nu$. Let $\Lambda$ be a set of positive measure for $\nu$. We say that a set $E \subset X_l^{\delta,n}$ is  \emph{$(n,\delta, \Lambda)$-separated} if it is \emph{$(n,\delta)$-separated} and if furthermore for all $W \in E$, $W\cap \Lambda \neq \varnothing$.
\begin{definition*}
	For $l \leq k$, $\nu$ an invariant probability measure and $\kappa>0$, we consider the quantity:
	 	\[ h_{\nu}^l(f,\kappa) :=\inf_{\Lambda, \ \nu(\Lambda) >1-\kappa} \varlimsup_{\delta \to 0} \varlimsup_{n\to \infty}  \frac{1}{n} \log \left(\max  \{\# E, \ E\subset X_l^{\delta,n} \ \mathrm{is} \ (n,\delta,\Lambda)-\mathrm{separated}  \} \right).
	 \]
	We define $h_{\nu}^l(f)$, the \emph{$l$-measure-theoretic entropy of $f$}, as the quantity:
	\[ h_{\nu}^l(f) :=\sup_{\kappa>0}  h_{\nu}^l(f,\kappa).
	\]
\end{definition*}
We show in Proposition~\ref{same_for_l=0} that  $ h_{\nu}^0(f)$  is the usual measure-theoretic entropy when $\nu$ is ergodic. As above, one has $l\mapsto  h_{\nu}^l(f)$ is decreasing and $ h_{\nu}^k(f)=0$. Our main result in that setting is the following.
\begin{theorem}\label{theorem_minoration}
	Let $f$ be a dominating meromorphic map of a compact Kähler manifold $X$. Let $\mu$ be a an ergodic invariant measure such that $h_\mu(f)>0$ and $\log \dist(., I) \in L^1(\mu)$. Assume that the (well-defined) Lyapunov exponents satisfy: 
	\[\chi_1 \geq \dots \geq \chi_s>0>\chi_{s+1} \geq \dots \geq  \chi_k. \] 
	Then, we have:
	\begin{equation}\label{eq_minoration}
	\forall l \leq k-s, \  h_{\mu}^l(f) = h_\mu(f).
	\end{equation}
\end{theorem}
Finally, in the particular case of endomorphisms of $\P^2$ of degree $d$, we show (generic is meant in the sense of \cite{thelin_genre}.):
\begin{theorem}\label{generic}
	Let $f$ be a generic holomorphic endomorphism of $\P^2(\C)$ of degree $d\geq 2$. Assume that $\mathrm{supp}(\mu) \neq \supp(T)$ where $T$ is the Green current of $f$ and $\mu:= T\wedge T$ is the measure of maximal entropy. Then :
	\[h_{\mathrm{top}}^1(f)= \log d. \]  
\end{theorem}
Observe that it is possible to have $h_{\mathrm{top}}^1(f) \neq \log d$ for specific holomorphic endomorphisms of $\P^2(\C)$. Indeed, we show in Subsection~\ref{Lattes} that for a Lattès map $f$ of $\P^2$ of degree $d$, then $h_{\mathrm{top}}^1(f)=0$ ($\mathrm{supp}(\mu) = \supp(T)$ in that case).

\section{Computing $h_{\top}^l(f)$ in terms of the dynamical degrees}

\subsection{Proof of Theorem~\ref{theorem_majoration}.}
	Take $\delta>0$ and let $E\subset X_l^{\delta,n}$ be $(n,\delta)$-separated with $\# E =N$. Consider a finite atlas of $X$. We can choose a chart $\Delta$ such that there are $c N$ elements $W$ of $E$ that intersects $\Delta$ where $c>0$ does not depend on $E$. Slightly enlarging $\Delta$ if necessary, we can assume that for each such $W\in E$, $W':=W\cap \Delta$ is in $  X_l^{\delta,n}$. Write $W'_1, \dots W'_{cN}$ the collection of those   $(n,\delta)$-separated sets. 
	
	 Let us consider the euclidean metric on $\Delta$, it is comparable with $\omega$ in $\Delta$ so we  can assume that each $W'_i$ is analytic  of exact dimension $l$ in $ B(x,c'e^{-n\delta})$ where $c'$ is a constant that depends only on $\Delta$.
	 
	 By the main result of \cite{AlexanderTaylor}, there exists a positive constant $C_l$ such that 
	 \[\sum_{\alpha} \lambda(\pi_\alpha(W'_i)) \geq C_l e^{-2ln\delta},\]
	  where the sum is over the $l$-dimensional coordinate planes $\alpha$ through $0$ in $\C^k$, $\pi_\alpha$ is the projection to $\alpha$, and $\lambda$ is the Lebesgue measure in $\mathbb{C}^l$ (notice that the area is counted \emph{without} the multiplicity). \\
	 
	 In particular, shrinking $c$ and $C_l$ if necessary, we can assume that they are $c N$ elements $W_i'\subset \Delta$ which are $(n,\delta)$-separated and a  $l$-dimensional coordinate plane $P$ 
	 such that the projection of each $W'$ on $P$ has volume $\geq  C_l e^{-2ln\delta}$. Observe that we can slightly move $P$ and the above still stands.  
	 
	 For the rest of the proof, we proceed as in \cite{DeThelin_inventiones}[p.110-111] so we only sketch the main steps. If $\pi_3:\Delta \to P$ denotes the canonical projection, we can assume that $\pi_3(\Delta)$ lies in a compact set $K$ of $P$. For $a\in K$, let $F_a:=\pi_3^{-1}(\{a\})$ and $da$ be the Lebesgue measure on $P$. Let $\mathcal{W}_s :=\cup_{i\leq cN} W'_i$ and let $n(a)$ be the number of intersection of $F(a)$ with $\mathcal{W}_s$ counted with multiplicity. That way:
	\begin{align*}
	\int n(a)da & = \mathrm{volume \ of\  the\  projection\ of\ \mathcal{W}_s\ on\ P }\\
                & \geq C_l c N e^{-2ln\delta}.
	 	\end{align*} 
Let $\Gamma_n(a)$ be the closure of $\{(z, f(z), \dots, f_{n-1}(z)), \  z \in F(a)\cap \Delta \}$ in $X^n$. It is the multigraph in $F(a)\cap\Delta$ that we endow with the Kähler form $\omega_n:=\sum_{i\leq n} \Pi_i^*(\omega)$ where $\Pi_i$ denotes the canonical projection of  $X^n$ on its $i$-th factor. We have \cite{DeThelin_inventiones}[Lemma 13]:
\[\int \mathrm{volume}(\Gamma_n(a)) da \geq c(\delta) \int n(a) da  \]
where $c(\delta)$ is a constant that depends only on $\delta$ (and not $n$). 

On the other, one shows that, for any $\varepsilon>0$, there exists $c_\varepsilon$ such that :
\[ \int \mathrm{volume}(\Gamma_n(a))da \leq c_\varepsilon n^{k-l} \max_{
0\leq j \leq k-l} (d_j + \varepsilon)^n.\]
In particular,
\[ c(\delta)C_l cN e^{-2ln\delta} \leq c_\varepsilon n^{k-l} \max_{
	0\leq j \leq k-l} (d_j + \varepsilon)^n.\]
We take the logarithm, divide by $n$ and  let $n\to \infty$. The result then follows by letting $\varepsilon \to 0$ and $\delta\to 0$. 
\begin{remark} \rm
	Observe that we do not require in the proof of Theorem~\ref{theorem_majoration} that the $l$-dimensional volume of the $W\in X^{\delta,n}_l$ is $\leq 1$. 
\end{remark}
	
	\subsection{Proof of Theorem~\ref{theorem_minoration}}
	Before proving Theorem~\ref{theorem_minoration}, we show that $h_{\nu}^0(f)$ is the usual measure-theoretic entropy when $\nu$ is ergodic.  Indeed, following \cite{Katok1980}, we have the following folklore's proposition whose proof will be useful for us (it can be easily extended to the case of meromorphic maps assuming $\nu(I)=0$). 
	\begin{proposition}\label{same_for_l=0}
		Assume that $\nu$ is ergodic, $f$ is holomorphic and $l=0$, then $h_{\nu}^0(f) = h_{\nu}(f)$ is the classical measure-theoretic entropy of $\nu$.
	\end{proposition}
	\begin{proof} 
		Observe first that $X_0^{\delta,n} =X$. Let $\varepsilon>0$ and $\Lambda \subset X$. Fix $1\gg\delta_0 >0$ and consider:
		\[X_{{\delta_0}, n}:= \left\{ x, \ \nu(B_n(x,{\delta_0})) \leq e^{-h_{\nu}(f)n +\varepsilon n}    \right\}. \]
		We know by Brin-Katok formula(\cite{BrinKatok} that
		\begin{equation*}
		1-\frac{\nu(\Lambda)}{4} \leq \nu \left(\left\{x, \ \underline{\lim} -\frac{1}{n} \log \nu(B_n(x,{\delta_0}))\geq h_\nu(f) -\frac{\varepsilon}{2} \right\}\right) \leq \nu \left(  \bigcup_{n_0} \bigcap_{n \geq n_0} X_{{\delta_0},n}   \right).
		\end{equation*} 
		In particular, we choose $n_0$ large enough so that:
		\[\nu\left(\bigcap_{n\geq n_0} X_{{\delta_0},n_0}\right) \geq 1-\frac{\nu(\Lambda)}{2}.\]
		Consider the set $\Lambda':= \Lambda \cap \bigcap_{n\geq n_0} X_{{\delta_0},n_0}$ which satisfies $\nu(\Lambda')>\nu(\Lambda)/2$ by construction. Then, for $n\geq n_0$, we start with a point $x_0 \in \Lambda'$ and we choose inductively a point $x_i \in \Lambda' \backslash \cup_{0\leq j <i} B_n(x_j,{\delta_0})$. This is possible as long as  $\nu(\Lambda') > \sum_{0\leq j <i} \nu(B_n(x_j,{\delta_0}))$. So using our hypothesis, we can find at least $N$ such points with $N$ given by:
		\[ N =  e^{ h_\nu(f)n- \varepsilon n} \nu(\Lambda').\] 
		In particular:
		\[\varlimsup_{\delta \to 0} \varlimsup_{n\to \infty}  \frac{1}{n} \log \left(\max  \{\# E, \ E\subset X \ \mathrm{is} \ (n,\delta,\Lambda)-\mathrm{separated}  \}\right) \geq h_\nu(f)-\varepsilon \]
		which gives the inequality $h_{\nu}^0(f) \geq h_{\nu}(f)$ by letting $\varepsilon \to 0$, taking the infimum over all $\Lambda$ with $\nu(\Lambda)>1-\kappa$ and letting $\kappa\to 0$. 
		
		For the other inequality, consider:
		\[Y_{k,{\delta_0}}:= \left\{ x, \ \forall n\geq k, \  \nu(B_n(x,{\delta_0/2})) \geq e^{-h_{\nu}(f)n - \varepsilon n 
		}    \right\}. \]
		Let $\kappa>0$, Brin-Katok formula implies that for $\delta_0$ small enough and $m$ large enough $\nu(\cap_{k\geq m} Y_{k, \delta_0})\geq 1 - \kappa$. In particular, we choose 
		\[\Lambda:= \cap_{k\geq m} Y_{k, \delta_0}.\] 
		Take $n\geq m$.  Consider $\{x_1,\dots, x_N\}$ a $(n,\delta_0, \Lambda)$-separated set. Then, the Bowen ball $B_n(x_i, \delta_0/2)$ are disjoint so:
		\[ N e^{-h_{\nu}(f)n - \varepsilon n 
		} \leq 1 .\]
		Taking the logarithm, dividing by $n$, letting $n\to \infty$ and $\delta_0 \to 0$ implies:
		\[ h_{\nu}^0(f) \leq  h_\nu(f) + \varepsilon
		\]
		and the result follows by letting $\varepsilon \to 0$.
	\end{proof}

	We now prove Theorem~\ref{theorem_minoration}.	Let $\mu$ be an ergodic invariant measure with $h_\mu(f)>0$. Assume that $\log \dist(., I) \in L^1(\mu)$. We recall some facts we need on Pesin theory in this setting \cite{DeThelin_inventiones}.  We shall use the results of that paper keeping the same notations: $\pi$, $\hat{X}$, $\hat{X}^*$, $\hat{\mu}$, $\tau_x$, $\varepsilon_0$, $f_x$, $f^n_x$, $f^{-n}_{\hat{x}}$, $D\hat{f}(\hat{x})$. We fix some some $\delta>0$:
	\begin{itemize}
		\item 	The Lyapunov exponent are well defined (Oseledec's Theorem). We assume that they satisfies 
		\[\chi_1 \geq \dots \geq \chi_s>0>\chi_{s+1} \geq \dots \geq \chi_k. \] 
		\item the set $\hat{Y}$ of points in the universal extension $\hat{X}$ of $X$ that satisfy the conclusion of Oseledec's Theorem, Pesin's Theorem and \cite{DeThelin_inventiones}[lemme 10] satisfies $\hat{\mu}(\hat{Y})=1$
		\item one can find a set $A$ of $\mu$-measure arbitrarily close to $1$ with $A\subset \pi(\hat{Y})$, an integer $n_0$ and a constant $\alpha_0>0$, such that (see \cite{DeThelin_inventiones}[p.103]):
		\begin{align*}
	\forall x \in A&, \  \forall n \geq n_0, \  \mu(B_n(x,2\delta)) \leq e^{-n h_\mu(f) +n \delta} .
	\end{align*} 
	\item for all $x$ in the set $A$, using iterated pull-backs and graph transforms of a suitable local complex $k-s$-plane, one can define an approximated stable manifold $W_n(x)$   such that:
	\begin{itemize}
		\item  $x\in W_n(x)$ and $W_n(x)\subset B(x, e^{-n\delta})$ is an analytic set of dimension $k-s$; 
		\item $W_n(x)$ is a graph over some $k-s$ plane of a $\alpha_0$-Lipschitz map;
		\item  $\mathrm{diam}(f^k(W_n(x))) \leq \exp(- n\delta)$ for all $k<n$. 
	\end{itemize}
	\end{itemize}
	In particular, for any $\Lambda$ with $\mu(\Lambda)>1-\kappa$, we can assume that $\Lambda \subset A$ (up to considering $\Lambda \cap A$ where $A$ is the above set of $\mu$-measure arbitrarily close to $1$).  Using the same volume argument as in the proof of Proposition~\ref{same_for_l=0}, we can thus find a $(n,\delta, \Lambda)$-separated set of cardinality $N\geq e^{n h_\mu(f) - n \delta} \mu(\Lambda)$ of points $x_1,\dots ,x_N$. 
	
	For $n$ large enough, we have that $\exp(-n\delta) \leq \delta /4$, in particular,  $(W_n(x_i))_{i\leq N}$ is a collection of sets of $X^{\delta/2, n}_{k-s}$ which are $(n, \delta/2, \Lambda)$-separated. Finally, as $W_n(x)$ is a graph over some $k-s$ plane of a $\alpha_0$-Lipschitz map, its volume is bounded by some constant that depends on $\alpha_0$ so it is $\leq 1$ for $n$ large enough.	
	Theorem~\ref{theorem_minoration} follows. 

\subsection{Computing $h_{\nu}^l(f)$ in some families of maps}
	\subsubsection{Holomorphic maps of $\P^k$}
Let $f:\P^k\to \P^k$ be a holomorphic map and take $0\leq l \leq k$. Assume that one can find an invariant ergodic probability measure $\mu$ of entropy $\log \max_{j\leq k-l }d_j= (k-l) \log d$ of saddle type for $f$:
\[\chi_1 \geq \dots \geq \chi_{k-l}>0>\chi_{k-l+1} \geq \dots \geq \chi_k, \] 
where $(\chi_i)_{i\leq k}$ are the well defined Lyapunov exponent of $\mu$. Then, as a consequence of Theorems \ref{theorem_majoration} and \ref{theorem_minoration}, one obtains directly:
\[ h_{\mathrm{top}}^l(f) = h_{\mu}^l(f)= (k-l) \log d .\]
In \cite{Dethelin_Fourier}, the first author constructed such measures for the case $k=2, l=1$ (with $\chi_2 \geq 0$ in general), see also \cite{Fornaess_Sibony_1998} for the initial case of hyperbolic maps. 

\subsubsection{Generic birational maps of $\P^k$} 
Let $f: \P^k \to \P^k$ be a birational maps such that $\mathrm{dim}(I(f))= k-s-1$ and $\mathrm{dim}(I(f^{-1}))= s-1$ for some $1\leq s\leq k-1$. Generalizing a construction of Bedford and Diller (\cite{Bedford_Diller}), we defined in \cite{ThelinVigny1} a condition on such maps under which we constructed a measure of maximal entropy $s\log d$ that integrates $\log \dist(., I)$ with 	$\chi_1 \geq \dots \geq \chi_{s}>0> \chi_{s+1} \geq  \dots \geq\chi_k$.  The condition is generic in the sense that for all $A$ outside a pluripolar set of $ \mathrm{Aut}(\P^k)$ and any $f$ such that  $\mathrm{dim}(I(f))= k-s-1$ and $\mathrm{dim}(I(f^{-1}))= s-1$, then $f \circ A$ satisfies the condition. Furthermore, the class of such generic birational maps contains the regular automorphisms of $\C^k$ and generic birational maps of $\P^k$ (\cite{Sibony, DinhSibonyregular}). By the above, for such a generic birational map of $\P^k$, one has:
\begin{equation*}
\forall l\leq k-s, \ h_{\mathrm{top}}^l(f) = s \log d.  
\end{equation*}

As a consequence, observe that one does not necessarily have $h_{\mathrm{top}}^l(f)=h_{\mathrm{top}}^l(f^{-1})$ for a birational map. Indeed, consider the case of a generic birational map of $\P^3$ whose dynamical degrees are $1=d_0<2=d_1<4=d_2>1=d_3$ (for example, one can take the regular automorphism $(z+y^2,y+x^2,x)$). Then, $h_{\mathrm{top}}^1(f)=\log 2$  and $h_{\mathrm{top}}^1(f^{-1})=\log 4$. Observe also that $h_{\mathrm{top}}^{k-l}(f)\neq h_{\mathrm{top}}^l(f^{-1})$ in general (for $k=l$, then $h_{\mathrm{top}}^{0}(f)=h_{\mathrm{top}}(f)\neq0 $ and  $h_{\mathrm{top}}^k(f^{-1})=0$).\\

Lastly, in \cite{moi_Fourier}, the second author generalized the construction of generic birational map to the rational case with no additional hypothesis on the dimension of the indeterminacy sets. In particular, he constructed saddle measures of maximal entropy under mild hypothesis, giving many examples where one can compute $h_{\mathrm{top}}^l(f)$ (for $l$ such that $d_l \leq \max d_j$).

\section{Proof of Theorem~\ref{generic}}

\subsection{Strategy of the proof and Yomdin's estimate}
Take $f$ as in Theorem~\ref{generic}. Consider a projective line $L$. By \cite{thelin_genre}, the genus of $f^{-n}(L)$ outside $\supp(\mu)$ is bounded by $d^ne^{\delta n}$. 

Using \cite{thelin_laminaire}, we will then construct approximately $d^ne^{\delta n}$ disks of size $e^{-\delta n}$ in $f^{-n}(L) \backslash \supp(\mu)$. Using a length-area argument, we will show that the size of those disks is still small when pushed forward by $f^i$, $i=0,\dots,n-1$. 

Finally, using Yomdin's theorem (\cite{yomdin}), we will construct a $(n,\delta)$-separated set from those disks. Here is the version of Yomdin's result we will use; it can be deduced from \cite{ThelinVigny1}[Proposition 2.3.2].
\begin{proposition}[Yomdin]\label{Yomdin}
Let $\delta'>0$. Then there exist $C_1>0$ and $\delta_0>0$ such that for all $0<\delta<\delta_0$, we have for any dynamical ball $B_n(x,5\delta)$ and any projective line $L$:
\[\forall n\in \N,   \left((f^n)^*\omega\wedge L\right)(B_n(x,5\delta)) \leq C_1e^{\delta' n} .    \] 	
\end{proposition}

\subsection{Finding a set with uniform estimates}
Let $\delta'>0$. Fix $0<\delta<\delta_0$ where $\delta_0$ is given by the above proposition. We fix $x_0\in \supp(T)\backslash \supp(\mu)$. There exist an open neighborhood $U$ of $\supp(\mu)$ and $0<r_0<1$ such that 
\[ B(x_0, r_0)\cap \overline{U}=\varnothing.  \]
Using \cite{thelin_genre}, we have:
\[\varlimsup_n  \frac{1}{n}\log \max_y \# \{z,  f^{n}(z)=y, \ z\notin U \}  \leq \log d  \]
and 
\[\varlimsup_n \frac{1}{n} \log \max_{L} \mathrm{Genus} f^{-n}(L) \backslash U  \leq \log d  \]
where the supremum is taken over the Grassmanian space $(\P^2)^*$ of complex projective lines in $\P^2$. Hence, there exist $n_2$ such that for $n \geq n_2$, 
\begin{align*}
\max_y \# \{z,  f^{n}(z)=y, \ z\notin U \}     &\leq  d^n e^{n \delta/3} \\
\max_{L} \mathrm{Genus} f^{-n}(L) \backslash U &\leq  d^n e^{n \delta/3}.
\end{align*}
In particular, there exists $C_2>0$ such that for all $n$, 
\begin{align*}
\max_y \# \{z,  f^{n}(z)=y, \ z\notin U \}     &\leq  C_2 d^n e^{n \delta/3} \\
\max_{L} \mathrm{Genus} f^{-n}(L) \backslash U &\leq   C_2 d^n e^{n \delta/3}.
\end{align*}
As $x_0 \in \supp(T)$, we have $T\wedge \omega (B(x_0,r_0/4))>0$. In particular, we can find a coordinate axis $D$ such that $T\wedge \pi^*(\omega_0)(B(x_0,r_0/4))>0$ (where $\pi$ denotes the orthogonal projection on $D$ and $\omega_0:=\omega_{|D}$). Now, as $T=\lim_n d^{-n} (f^n)^*(\omega)$, there exists $\varepsilon_1>0$ and $n_3 \in \N$ such that 
\begin{equation}\label{page5}
\forall n\geq n_3,  \ d^{-n} (f^n)^*(\omega)\wedge \pi^*(\omega_0)(B(x_0,r_0/2))\geq \varepsilon_1. 
\end{equation}
In what follows, we take $n\geq n_3$.

\subsection{Constructing disks of size $e^{-\delta n}$ in $f^{-n}(L)$}
We follow \cite{thelin_laminaire}. We subdivide the square $C_0 \subset D$, centered at $x_0$ and size $2$, into $4k^2$ identical squares with $k=e^{\delta n}/4$. Such subdivision contains four families of $k^2$ disjoint squares that we will denote by $Q_1$, $Q_2$, $Q_3$ and $Q_4$. Let $V$ denote the Fubini-Study measure of $(\P^2)^*$ normalized so that $\omega=\int [L] dV(L)$.

Up to moving slightly the subdivision, we assume that
\[ V(\{L,  \ \pi_{|f^{-n}(L)} \  \mathrm{has \ a \  critical \ value\ in} \ \partial Q_i \  \mathrm{ for \ some \ } i    \}) =0  \]
(indeed, for each $L$ there exists a finite number of critical value and we conclude using Fubini).  
Let $r_0/2<\rho<r_0$. We fix $L$  such that $\pi_{|f^{-n}(L)}$ has no critical value in $\partial Q_i$ for all $i$. We follow  \cite{thelin_laminaire} keeping the more precise estimates we need. 

We start with the geometric simplification of $C_n:= f^{-n}(L)$ (\cite{thelin_laminaire}[Paragraph 2.1]). Let 
\[\rho'\in [\rho+\frac{r_0 - \rho}{4}, r_0 - \frac{r_0 - \rho}{4} ] .\]
We denote by $\widetilde{C}_n$ the complex curve obtained by taking the union of $C_n \cap B(x_0,\rho)$ with the connected components of $C_n \cap \left( B(x_0, \rho')\backslash B(x_0, \rho)  \right)$ that meet $\partial B(x_0,\rho)$. By the maximum principle,  the boundary of $\widetilde{C}_n$ is contained in $\partial B(x_0,\rho')$ and if $\widetilde{B}_n$ denotes the number of its boundary components, we have by \cite{thelin_laminaire} that
\[ (\widetilde{B}_n -G_n) \left(\frac{r_0-\rho}{4} \right)^2 \leq A_n:= \mathrm{Area \ of } \ f^{-n}(L) \ \mathrm{ in}\ B(x_0,r_0) \leq d^n,  \]
where $G_n$ is the genus of $f^{-n}(L)$ in $B(x_0,r_0) $. Furthermore, by the coarea formula and letting $\rho'$ moves in $[\rho+\frac{r_0 - \rho}{4}, r_0 - \frac{r_0 - \rho}{4} ]$, we can find $\rho'$ such that $\widetilde{L}_n$, the length of the boundary of $\widetilde{C}_n$, satisfies
\[\widetilde{L}_n \leq \frac{2A_n}{r_0-\rho}. \]
Notice that $\rho'$ depends on $L$ but not the initial $\rho$. In addition, $\widetilde{C}_n$ coincides with $C_n$ in $B(x_0, \rho)$ by construction. Summing up, if $\widetilde{G}_n$ is the genus of   $\widetilde{C}_n$, we have
\begin{align*}
\widetilde{G}_n \leq G_n \leq C_2 d^n e^{n\delta/3}, \quad 
\widetilde{B}_n\leq \left(\frac{4}{r_0-\rho}\right)^2 d^n+  C_2 d^n e^{n\delta/3}, \quad 
\widetilde{L}_n \leq d^n \frac{2}{r_0-\rho}.  
\end{align*}
We now continue with the idea of \cite{thelin_laminaire}[Paragraph 2.2]. We fix a family  $Q$ of squares amongst the $(Q_i)_{i\leq 4}$. We can tile $C_0\backslash Q$ by crosses. Let $\Sigma$ be a component above a cross (for $\pi$). If $l(\Sigma)$ denotes the length of the relative boundary of $\Sigma$ and $a(\Sigma)$ the area of $\pi(\Sigma)$ counted with multiplicity, then we have, taking  $\varepsilon_k= e^{-2\delta n/3} $:
\[l(\Sigma)\geq \frac{\varepsilon_k}{k} = 4 e^{-\delta n } e^{-2\delta n/3} \quad \mathrm{or } \quad  l(\Sigma) \leq \frac{\varepsilon_k}{k} = 4 e^{-\delta n } e^{-2\delta n/3}. \]  
Using the isoperimetric inequality in the latter case implies
\[a(\Sigma)\geq (1-\varepsilon_k) \times \mathrm{Area \ of \ the \  cross }=\frac{ 3(1-\varepsilon_k) }{k^2}  \quad \mathrm{or } \quad a(\Sigma)\leq l(\Sigma)^2.  \]
As in \cite{thelin_laminaire}, we want to remove the components of that latter type. Let us denote them by $A_1, \dots, A_m$. If we remove them from $\widetilde{C}_n$, we change its area for $\pi^*(\omega)$ by at most 
\[\sum_{i=1}^m l_i^2\leq  \frac{1}{k} \sum  l_i \leq \frac{\tilde{L}_n}{k}.  \]
By the triangular inequality, the length of the relative boundary of the curve obtained by removing those components is $\leq 2 \widetilde{L}_n$. Finally, this curve still satisfies: 
\begin{align*}
\widetilde{G}_n \leq C_2 d^n e^{n\delta/3} \ \mathrm{and} \
\widetilde{B}_n\leq \left(\frac{4}{r_0-\rho}\right)^2 d^n+  C_2 d^n e^{n\delta/3}.
\end{align*}
We still denote by $\widetilde{C}_n$ that curve and we proceed with \cite{thelin_laminaire}[Paragraph 2.3]. Let $\mathcal{I}$ denote the set of islands above $Q$. We construct a graph where each vertex is a connected component above a cross and where we put as many edges between two vertices as the corresponding components share common arcs. Then, by \cite{thelin_laminaire}[Paragraph 1.2]
\[\# \mathcal{I}\geq \chi(\widetilde{C}_n)-s+a,  \]
where $s$ is the number of vertices and $a$ the number of edges of the graph. We now bound $s$ from above and $a$ from below, starting with $s$.

There are at most $ \frac{2\widetilde{L}_n k}{\varepsilon_k}$ vertices $\Sigma$ satisfying $l(\Sigma) \geq \frac{\epsilon_k}{k}$. Furthermore, the cardinality of vertices such that $a(\Sigma)\geq (1-\varepsilon_k)\frac{3}{k^2}$ is bounded from above by the area of the projection of $C_0 \backslash Q$ (counted with multiplicity) divided by $(1-\varepsilon_k)\frac{3}{k^2}$ hence it is bounded by
\[ \frac{3 S_n(C_0 \backslash Q )}{(1-\varepsilon_k)\frac{3}{k^2}} = \frac{k^2 S_n(C_0 \backslash Q )}{(1-\varepsilon_k)}  \]
where $S_n(C_0 \backslash Q )$ is the mean number leaves above $C_0 \backslash Q$. We thus have:
\[s \leq \frac{2\widetilde{L}_n k}{\varepsilon_k}+ \frac{k^2 S_n(C_0 \backslash Q )}{(1-\varepsilon_k)} . \]
We now bound $a$ from below. Following line by line \cite{thelin_laminaire}[Paragraph 2.3] gives:
\[a=\frac{1}{2} \sum_{\mathrm{vertices}} v(\Sigma) \geq 2 k^2 S_n(C_0 \backslash Q)-4hk\widetilde{L}_n. \]
Combining those bounds implies the following lower bound on $\# \mathcal{I}$
\begin{align*}
\# \mathcal{I} &\geq \chi(\widetilde{C}_n)- \frac{2\widetilde{L}_n k}{\varepsilon_k}- \frac{k^2 S_n(C_0 \backslash Q )}{(1-\varepsilon_k)}  + 2 k^2 S_n(C_0 \backslash Q)-4hk\widetilde{L}_n    \\
               &\geq   -2 \widetilde{G}_n-\widetilde{B}_n+ k^2 S_n(C_0 \backslash Q)(2-\frac{1}{1-\varepsilon_k}) -8h \frac{k}{\varepsilon_k}\widetilde{L}_n
\end{align*}
(we can always assume $h\geq 1$). Now, if we denote by $I_n(L)$ the total number of islands above the four families of squares $Q_1, Q_2, Q_3$ and $Q_4$, we have:   
\begin{align*}
I_n(L)&\geq   -8 \widetilde{G}_n-4\widetilde{B}_n-32h \frac{k}{\varepsilon_k}\widetilde{L}_n  + k^2 (2-\frac{1}{1-\varepsilon_k}) \sum_{i=1}^4 S_n(C_0 \backslash Q_i) \\
       & \geq  -8 \widetilde{G}_n-4\widetilde{B}_n-32h \frac{k}{\varepsilon_k}\widetilde{L}_n  + k^2 (2-\frac{1}{1-\varepsilon_k})4 S_n  
\end{align*}
where $S_n$ is the mean number of leaves above $C_0$. Let $a_n(L)$ denote the area of $\widetilde{C}_n$ in $B(x_0, \rho)$ for $\pi^*(\omega)$. We have $4S_n \geq a_n(L)- \frac{\widetilde{L}_n}{k} $ (we removed $A_1, \dots A_m$). Hence
\begin{align*}
I_n(L)&\geq   -12 C_2 d^ne^{ n \delta/3} -4\left( \frac{4}{r_0-\rho} \right)^2 d^n- \frac{32h k d^n}{\varepsilon_k} \frac{2}{r_0-\rho} \\
 &  \quad + k^2 ( 2-(1+2\varepsilon_k))( a_n(L)- \frac{2 d^n}{(r_0-\rho)k}   ) \\
     &\geq -h \frac{d^ne^{5 \delta n /3}}{(r_0-\rho)^2}C_2'+(1-2\varepsilon_k)k^2a_n(L)
\end{align*}
where we used that $\varepsilon_k=\exp(-2\delta n /3)$, $k= \exp(\delta n)/4$ and where $C_2'$ is a constant independent of $n$ and $L$. 

Amongst those islands, few are ramified. Indeed, assume $N_1$ of them are ramified, for those islands $\Delta$, the area of $\pi(\Delta)$ is $\geq 2\frac{1}{k^2}$ so
\begin{align*}
a_n(L) &\geq \frac{2}{k^2}N_1 +\frac{1}{k^2}(I_n(L)-N_1)= \frac{N_1}{k^2} + \frac{I_n(L)}{k^2} \\  
       &\geq \frac{N_1}{k^2} +(1-2\varepsilon_k)a_n(L)
        -h \frac{d^ne^{5 \delta n /3}}{k^2(r_0-\rho)^2}C_2'
\end{align*}
so
\[
N_1\leq 2\varepsilon_k k^2a_n(L)
+h \frac{d^ne^{5 \delta n /3}}{(r_0-\rho)^2}C_2'.
\]
Finally the number of islands of volume $\geq 1$ is at most $d^n$. In particular, if $I_n'(L)$ denotes the number of unramified islands of volume $\leq 1$, we have, replacing $h$ by $h C'_2$ if necessary.
\begin{equation}
\label{Numberofgoodislands}
I_n'(L) \geq (1-4\varepsilon_k)k^2 a_n(L) -2h \frac{d^ne^{5 \delta n /3}}{(r_0-\rho)^2} .
\end{equation}

\subsection{Control of size the image by $f^i$ of the above disks}\label{step2}
We now want to construct, in the above good islands, many disks whose size stay small when 
we push them forward by $f^i$ for $i=0,\dots, n-1$. Let $q$ be a square in the above family. In $q$, we consider the annulus $A:= D(t, \frac{1}{2k})\backslash D(t, \frac{1}{4k})$, where $t$ is the center of $q$: 
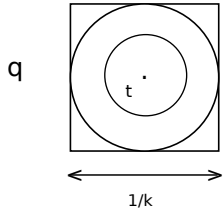
	\begin{center}
		\includegraphics[scale=0.2]{dessin} 
	\end{center}
Let $i\in \{0,\dots, n-1\}$ and $\Delta$ a good island above $q$. We shall use a length-area argument. The form $\omega_{|f^i(\Delta)}$ defines a metric $\beta$, we take the pull-back of this metric by $f^i$ then we push it forward by $\pi$ (it is a biholomorphism on $\Delta$). This gives a conformal metric $h_0=\sigma |dz|$ on $q$. Then
\begin{align*}
\mathrm{Area}_{h_0}(A)&= \int_{\frac{1}{4k}}^{\frac{1}{2k}} \int_0^{2\pi} \sigma^2(re^{i\theta})r d\theta dr \\
                      &\geq \int_{\frac{1}{4k}}^{\frac{1}{2k}}   \frac{1}{2\pi} \left( \int_0^{2\pi} \sigma(re^{i\theta})r d\theta \right)^2 dr\\
                      &\geq \frac{\log 2}{2\pi} \inf_{\gamma\in \Gamma} (l_{h_0}(\gamma))^2
\end{align*}
 where $ \Gamma$ is the set of circles of center $t$ and radius in $[1/(4k),1/(2k)]$ (essential curves of $A$). It implies the existence of an essential curve $\gamma$ such that
 \[\left( l_{h_0(\gamma)}    \right)^2\leq \frac{2\pi}{\log 2} \mathrm{Area}_{h_0}(A)=  \frac{2\pi}{\log 2} \mathrm{Area}_{\omega}(f^i(\Delta))  \]
where the area is counted with multiplicity and where $l_{h_0(\gamma)} $ is the length of $f^i(\pi_{|\Delta}^{-1}(\gamma))$ for the metric $\omega_{|f^i(\Delta)}$ i.e. $\omega_{|f^{-n+i}(L)}$.
 
 Let $\tilde{\Delta}$ be the part of $\Delta$ above the disk $D(t, 1/4k)$. 
By the Appendix of \cite{briendduval2}, we have:
\begin{align*} \mathrm{diameter \ of\ } f^i(\tilde{\Delta}) \ \mathrm{for\ }  \omega_{|f^{-n+i}(L)}&\leq \mathrm{diameter \ of\ } f^i( (\pi_{|\Delta})^{-1}(D_\gamma)) \ \mathrm{for\ }  \omega_{|f^{-n+i}(L)}                  \\
               &\leq 3 \sqrt{\frac{2\pi}{\log 2}} \sqrt{\mathrm{Area \ of\ } f^i(\Delta) \mathrm{\ counted \ with \ multiplicity}}
\end{align*}
where $D_\gamma$ is the disk in $q$ delimited by $\gamma$. 

We denote by $\Delta_1, \dots \Delta_M$ the $I'_n(L)$ good islands constructed at the first step and  by $\tilde{\Delta}_1, \dots \tilde{\Delta}_M$ the corresponding $\tilde{\Delta}$. The $\Delta_j$ are in $B(x_0,r_0)$ thus not in $U$. Since $\max_y \# \{z,  f^{i}(z)=y, \ z\notin U \}   \leq  C_2 d^i e^{i \delta/3} $ for all $i$, we have that the $f^i(\Delta_j)$ may recover themselves at most $C_2 d^i e^{i \delta /3}$. So the area of $f^i(\cup \Delta_j)$ is thus bounded from above by the area of $f^{-n+i}(L)$  (which is $d^{n-i}$) times $C_2 d^i e^{i \delta /3}$. In particular, the number of $\Delta_j$ such that the area of $f^i(\Delta_j)$ (counted with multiplicity) is greater than $\frac{\delta^2}{18\pi} \log 2$ is $\leq C_2 d^n \frac{e^{\delta i /3} 18\pi}{ \delta^2  \log 2}$ . If we remove those disks for $i=0,\dots, n-1$, we remove at most:
\begin{align*} d^n \frac{C_2 18 \pi}{ \delta^2 \log 2  }\sum_{i=0}^{n-1} e^{ i \delta /3} \leq d^n \frac{e^{n \delta/3}}{e^\delta -1}  C_2 \frac{18 \pi}{  \delta^2 \log 2 }= C_3 d^n e^{\delta n/3}
\end{align*}
where $C_3>1$ is another constant that does not depend on $n$ nor $L$. Using \eqref{Numberofgoodislands}, we deduce that the number $I''_n(L)$ of good islands constructed at the first step for which the area of $f^i(\Delta_j)$ (counted with multiplicity) is $\leq \frac{\delta^2}{18\pi} \log 2$ for all $i\leq n-1$ satisfies
\begin{equation}
\label{Numberofgoodislands2} 
I''_n(L) \geq I'_n(L)- C_3 d^n e^{\delta n/3} \geq (1-4\varepsilon_k)k^2 a_n(L) -3hC_3 \frac{d^ne^{5 \delta n /3}}{(r_0-\rho)^2}. 
\end{equation}
We denote $\Delta_1, \dots, \Delta_{M'}$ these islands and $\tilde{\Delta}_1, \dots, \tilde{\Delta}_{M'}$ the corresponding $\tilde{\Delta}$ (with our notations, $M'=I''_n(L)$). By construction, for each $1\leq j\leq M'$, $0\leq i \leq n-1$:
\[\mathrm{diam} f^i(\tilde{\Delta}_j) \leq 3 \sqrt{\frac{2\pi}{\log 2}} \frac{\delta}{\sqrt{18\pi} }\sqrt{\log 2}=\delta. \]

\subsection{Using Yomdin's estimate to produce $(n,\delta)$-separated disks}
Denote $\widetilde{f^{-n}(L)}$ the part of $f^{-n}(L)$ made with the $\tilde{\Delta}_1, \dots, \tilde{\Delta}_{M'}$. By construction, the area of $\pi(\widetilde{f^{-n}(L)})$  is
\begin{align*}
\pi \left(\frac{1}{4k}\right)^2 I''_n(L) &\geq   \frac{\pi}{16} (1-4\varepsilon_k)a_n(L) -\frac{\pi}{16k^2} 3hC_3 \frac{d^ne^{5 \delta n /3}}{(r_0-\rho)^2}    \\ 
                                         &=  \frac{\pi}{16} (1-4e^{-2 \delta n/3})a_n(L) -\pi 3hC_3 \frac{d^ne^{-\delta n /3}}{(r_0-\rho)^2}   
                                       \end{align*}
since $4k=e^{\delta n}$ and $\varepsilon_k=e^{-2\delta n/3}$ by definition. Using Proposition~\ref{Yomdin}, we have for every dynamical ball $B_n(x , 5\delta)$
\[  (f^n)^*\omega \wedge \pi^*\omega_0(B_n(x , 5\delta))\leq C_1 e^{\delta' n}            \]
thus (recall that $V$ denote the Fubini-Study measure of $(\P^2)^*$)
\[\int [\widetilde{f^{-n}(L)} ] \wedge \pi^*(\omega_0)  (B_n(x , 5\delta)) d V(L) \leq   \int [f^{-n}(L) ] \wedge \pi^*(\omega_0)  (B_n(x , 5\delta)) d V(L)  \leq C_1e^{\delta' n}. \]
Furthermore, integrating the previous estimate and using \eqref{page5}, we have for $n $ large enough
\begin{align*}
 \int [\widetilde{f^{-n}(L)} ] \wedge \pi^*(\omega_0)( \P^2(\C)) dV(L) &\geq   \frac{\pi}{32} \int a_n(L) dV(L) -\pi 3hC_3 \frac{d^ne^{-\delta n /3}}{(r_0-\rho)^2} \\
                                                       &\geq    \varepsilon_1 \frac{\pi}{32}d^n- \pi 3hC_3 \frac{d^ne^{-\delta n /3}}{(r_0-\rho)^2}
                                                       \geq \varepsilon_1 \frac{\pi}{64} d^n.
\end{align*}
In particular, we can find points $x_1, \dots, x_{N_1}$ with 
\[N_1\geq  \varepsilon_1 \frac{\pi}{64 C_1} d^n e^{-\delta'n} \]
which are $(n,5\delta)$-separated and belong to the support of  $\int [\widetilde{f^{-n}(L)} ] \wedge \pi^*(\omega_0) dV(L)$. In particular, for each $x_j$
  \[\int [\widetilde{f^{-n}(L)} ] \wedge \pi^*(\omega_0)(B_n(x_j,\delta)) dV(L)>0, \]
so for each $x_j$ we can choose a disk $\widetilde{\Delta}_j$ in $f^{-n}(\widetilde{L_j})$ ($L_j$ might depends on $j$) which meets $B_n(x_j, \delta)$. 
\begin{lemma}
	The disks $\widetilde{\Delta}_1, \dots, \widetilde{\Delta}_{N_1}$ are $(n,\delta)$-separated. 
\end{lemma}
\begin{proof}
	Let $i\neq j$, $x_i$ and $x_j$ are $(n,5\delta)$-separated so there exists $l\leq n-1$ such that $d(f^l(x_i), f^l(x_j))\geq 5 \delta$. In subsection~\ref{step2}, we showed that $\mathrm{diam}(f^l(\widetilde{\Delta}_i)) \leq \delta$ and   $\mathrm{diam}(f^l(\widetilde{\Delta}_j)) \leq \delta$. In particular, $f^l(\widetilde{\Delta}_i) \subset B(f^l(x_i), 2\delta)$ (since $\widetilde{\Delta}_i$ intersects $B_n(x_i,5\delta)$) and $f^l(\widetilde{\Delta}_j) \subset B(f^l(x_j), 2\delta)$. Thus $\mathrm{dist}( f^l(\widetilde{\Delta}_i), f^l(\widetilde{\Delta}_j))\geq 5\delta - 2\delta-2\delta =\delta$.	
\end{proof}
The $\widetilde{\Delta}_j$ are $(n,\delta)$-separated disks and they are graphs above a disk of the type $D(0, e^{-\delta n })$ and the volume of each $\widetilde{\Delta}_j$ is $\leq 1 $ by the above. In conclusion, for $n$ large enough
\[
 \frac{1}{n} \log \left(\max  \{\# E, \ E\subset X_1^{\delta,n} \ \mathrm{is} \ (n,\delta)-\mathrm{separated}  \} \right) \geq \frac{1}{n} \log \left(\varepsilon_1 \frac{\pi}{64 C_1}\right) +\log d -\delta'.
\]
so 
\[ \varlimsup_{n\to \infty}  \frac{1}{n} \log \left(\max  \{\# E, \ E\subset X_1^{\delta,n} \ \mathrm{is} \ (n,\delta)-\mathrm{separated}  \} \right) \geq \log d -\delta'.
\]
As this is true $\forall \delta< \delta_0$, we can make $\delta \to 0$ in the above inequality so
\[ h_{\mathrm{top}}^1(f) \geq  \log d - \delta'.
\]
This is what we want as $\delta'$ can be taken arbitrarily small and $h_{\mathrm{top}}^1(f) \leq \log d$ by Theorem~\ref{theorem_majoration}. Theorem~\ref{generic} follows.

\subsection{$h_{\mathrm{top}}^1(f)=0$ for Lattès maps}\label{Lattes}
As mentioned in the introduction, the purpose of this section is to show that   $h_{\mathrm{top}}^1(f) =0 $ for a Lattès map $f$ of $\P^2$ of degree $d$. 
 Such map can be lifted to an affine map $D$ of linear part $\sqrt{d}U$, where $U$ is an isometry, on a torus $\mathbb{T}$ by a holomorphic branch cover $\sigma$ (\cite{Dupontlattes}). 
 
By contradiction, assume that $h_{\mathrm{top}}^1(f) =\alpha>0$. For $\delta>0$ small enough, we thus have
\begin{equation}
\label{ifnot} 
\varlimsup_{n\to \infty}  \frac{1}{n} \log \left(\max  \{\# E, \ E\subset X_1^{\delta,n} \ \mathrm{is} \ (n,\delta)-\mathrm{separated}  \} \right) \geq \frac{\alpha}{2}.
\end{equation} 

\begin{claim} 
For all $ n$ such  that 
	 \[\frac{1}{n} \log \left(\max  \{\# E, \ E\subset X_1^{\delta,n} \ \mathrm{is} \ (n,\delta)-\mathrm{separated}  \} \right)   \geq \frac{\alpha}{4}, \] 
	 there exist $\psi(n)$ with  \  $\frac{\alpha n}{9\log d} \leq \psi(n) \leq n$ and  $\Delta^1_n, \Delta^2_n \in  X_1^{\delta,n}$ such that:
\[d\left( f^{\psi(n)}(\Delta^1_n),f^{\psi(n)}(\Delta^2_n)  \right) \geq \delta 	.\]	
\end{claim}
Indeed, if the claim does not hold then we could take arbitrarily large $n$ for which the disks in $X_1^{\delta,n}$ are not $\delta$-separated for $f^{m}$ for any $\frac{\alpha n}{9\log d} \leq m \leq n $. By \eqref{ifnot}, we have arbitrarily large $n$ such that 
\[ \max  \{\# E, \ E\subset X_1^{\delta,n} \ \mathrm{is} \ (n,\delta)-\mathrm{separated}  \} \  \geq e^{\frac{\alpha n}{4}}  . \]
Thus for such $n$ and such $E$ realizing the above maximum, the $\delta$-separation between the elements of $E$ is achieved for $m \leq \frac{\alpha n}{9\log d}$. For $W\in E$, we pick $x_W \in W$. Then observe that the collection of such $(x_W)_{W\in E}$ defines a $(\frac{ \alpha n}{9\log d},\delta)$-separated set which has the same cardinality than $E$. Thus:
\begin{align*}
\varlimsup  \frac{9\log d}{ \alpha n}  \log \max\{ \# E',\ E'\subset \P^2 \ \mathrm{is} \ \left(\frac{ \alpha n}{9\log d},\delta\right)-\mathrm{separated} \} &\geq \frac{\alpha}{4} \frac{9\log d}{\alpha}\\
  &> 2\log d.
\end{align*}
This contradicts the fact that $h_\top(f)\leq 2\log d$. \\

By the claim, there exists a sequence $(n_k)_k$ of integers with $n_k\to \infty$ such that for all $k$ we can find two elements 
$\Delta^1_{n_k}, \Delta^2_{n_k} \in  X_1^{\delta,n_k}$ such that:
\[d\left( f^{\psi(n_k)}(\Delta^1_{n_k}),f^{\psi({n_k})}(\Delta^2_{n_k})  \right) \geq \delta.\] 
For each $i$, there exists $Z_{i,n_k}\in \Delta^i_{n_k}$ such that $\Delta^i_{n_k}$ is analytic in $B(Z_{i,n_k}, e^{-\delta  n_k})$. Let $\hat{\Delta}^i_{n_k}$ be the restriction of $\Delta^i_{n_k}$ to the ball $B(Z_{i,n_k}, e^{-\delta  n_k}/2)$.  

Consider the lifts $\sigma^{-1}(\hat{\Delta}^i_{n_k})$ for $i=1,2$. Since $\sigma$ is holomorphic, $\sigma^{-1}(\hat{\Delta}^i_{n_k})$ is analytic in some ball $B(x_{i,n_k}, C e^{-{n_k}\delta})$ for some constant $C>0$ that depends on $\sigma$ (but not on ${n_k}$) and where $x_{i,n_k} \in \sigma^{-1}(\hat{\Delta}^i_{n_k})$. We denote $ T^i_{n_k}:= \sigma^{-1}(\hat{\Delta}^i_{n_k})\cap B(x_i, C e^{- {n_k}\delta})$.

We want to bound from above the volume of  $ T^i_{n_k}$. Let $\Omega$ be the metric preserved by $U$ and $\omega_{\mathrm{FS}}$ the Fubini-Study form on $\P^2$. Then, up to multiplying $\Omega$ by a suitable constant, the Green current $T^+$ of $f$ is given by $\sigma_*(\Omega)$ (see \cite{BertelootLoeb}[Proposition 4.1]) and it can be written as $T^+=\sigma_*(\Omega) = \omega + dd^cg$ where $g$ is a $\beta$-Hölder function (the Green current). Then
\[\mathrm{Vol}_\Omega (T^i_{n_k})=\int_{ T^i_{n_k}} \Omega \leq  \int_{\hat{\Delta}^i_{n_k}} \sigma_*\Omega.              \]
Let $\theta_{i,n_k}$ be a smooth cut-off function equal to $1$ in $B(Z_{i,n_k}, e^{-\delta  n_k}/2)$ and $0$ outside $B(Z_{i,n_k}, e^{-\delta  n_k})$ such that $0 \leq  \pm dd^c \theta_{i,n_k} + C e^{2 \delta n_k}\omega $ (changing $C$ if necessary). In particular, by Stokes and the definition of $X_1^{\delta,n_k}$
\begin{align*}
\mathrm{Vol}_\Omega ( T^i_{n_k}) &\leq  \int_{\Delta^i_{n_k}} \theta_{i,n_k} (\omega+dd^c g ) \\
  &\leq   \int_{\Delta^i_{n_k}} \omega + \int_{\Delta^i_{n_k}} (g-g(Z_{i,n_k})) dd^c\theta_{i,n_k}   \\
  &\leq 1+ C e^{-\beta \delta n_k } e^{2 \delta n_k} .
\end{align*} 
In particular, for a suitable $p>0$, we have that $T^i_{n_k}$ has volume $\leq C e^{p\delta n_k}$ (again, changing $C$ if necessary).

Let $\rho_{i,{n_k}}$ be a smooth cut-off function equal to $1$ in $B(x_{i,n_k}, \frac{C}{2} e^{- {n_k} \delta})$ and equal to $0$ outside $B(x_{i,n_k}, C e^{- {n_k}\delta})$. We can define it so it satisfies
 \[ \sqrt{-1}\partial \rho_{i,{n_k}} \wedge \bar{\partial} \rho_{i,{n_k}} \leq C' e^{2\delta n_k} \Omega   \]
  where $C'$ is another constant that does not depends on ${n_k}$. For each $i$, we define the positive $(1,1)$ current $S^i_{n_k}$, using the above notation
\[ S^i_{n_k} := \frac{\rho_{i,{n_k}}}{a_{i,n_k}} \left[  T^i_{n_k} \right]\]
where $a_{i,n_k}$ is a constant chosen so that $ S^i_{n_k} $ has mass $1$. By Lelong's inequality, we know that $a_{i,n_k} \geq e^{-2 n_k \delta } C^2\pi/4 $.

 We claim that, up to extracting, 
\[ d^{-\psi(n_k)} (D^{\psi(n_k)})_*( S^i_{n_k} ) \to R_i\]
where $R_i$ is a positive closed $(1,1)$ current of mass $1$ in $\mathbb{T}$. As $\Omega$ is the metric preserved by $U$, then $d^{-1}D^*(\Omega)=\Omega$ thus  $d^{-\psi(n_k)} (D^{\psi(n_k)})_*( S^i_{n_k} ) $ has mass $1$ and we can extract a converging subsequence for both $i=1$ and $i=2$. To show that the limit $R_i$ of such subsequence is closed, it is enough to show that it is $\partial$-closed so we only have to test against forms of the type $\chi \bar{\partial} \bar{Z}$ where  $\chi$ is a smooth function and $\bar{\partial} \bar{Z}$ some $(0,1)$-form with constant coefficients and norm $1$. As $D$ has linear part $\sqrt{d}U$, we have that $ (D^{\psi(n_k)})^*(\bar{\partial} \bar{Z})$ is again some $1$-form with constant coefficients and norm $ d^{\psi(n_k)/2}$ so we write it as $ d^{\psi(n_k)/2} \bar{\partial} \bar{Z}_{n_k} $. Now, 
\begin{align*}
  \left \langle \partial( d^{-\psi(n_k)} (D^{\psi(n_k)})_*( S^i_{n_k} )), \chi \bar{\partial} \bar{Z} \right\rangle =   &\frac{d^{-\psi(n_k)} }{a_{i,n_k}}\left \langle  \partial\rho_{i,{n_k}} \wedge \left[  T^i_{n_k} \right], \chi\circ D^{\psi(n_k)}  (D^{\psi(n_k)})^*(\bar{\partial} \bar{Z}) \right\rangle \\
 =  &\frac{d^{-\psi(n_k)/2} }{a_{i,n_k}}\left\langle \left[  T^i_{n_k} \right], \chi\circ D^{\psi(n_k)} \partial\rho_{i,{n_k}} \wedge \bar{\partial} \bar{Z}_{n_k} \right\rangle .
  \end{align*}
By Cauchy-Schwarz inequality and the properties of $\rho_{i,{n_k}}$, we deduce
\begin{align*}
\left|\left\langle \partial( d^{-\psi(n_k)} (D^{\psi(n_k)})_*( S^i_{n_k} )), \chi \bar{\partial} \bar{Z} \right\rangle\right|^2 &\leq \frac{d^{-\psi(n_k)} }{a^2_{i,n_k}}\left\langle  \left[  T^i_{n_k} \right], \sqrt{-1}\partial \rho_{i,{n_k}} \wedge \bar{\partial} \rho_{i,{n_k}} \right\rangle \\
  \ \ &\times \left\langle  \left[  T^i_{n_k} \right], |\chi\circ D^{\psi(n_k)}|^2 \sqrt{-1}\partial Z_{n_k} \wedge \bar{\partial} \bar{Z}_{n_k} \right\rangle\\
  &\leq \frac{\|\chi\|^2_\infty d^{-\psi(n_k)} }{a^2_{i,n_k}}\left\langle \left[  T^i_{n_k} \right],C' e^{2 \delta {n_k}} \Omega\right\rangle \times \left\langle \left[  T^i_{n_k} \right],  \Omega \right\rangle\\
   &\leq C''\|\chi\|^2_\infty d^{-\psi(n_k)+ (6+2p) \delta n_k} \\
   &\leq  C''\|\chi\|^2_\infty d^{\left(-\frac{\alpha }{9 \log d}+ (6+2p) \delta\right) n_k} 
\end{align*}
where $C''$ is another constant that does not depends on $n_k$. In particular, for $\delta$ small enough that quantity converges to $0$ so $R_1$ and $R_2$ are closed. 

By Bézout, we can take $Z \in \supp(\sigma_*(R_1))\cap \supp(\sigma_*(R_1))$. Take $\theta$ a smooth non zero form with support in $B(Z,\delta/4)$ and so that  $\langle \sigma_*(R_1), \theta \rangle \neq 0$ and $\langle \sigma_*(R_2), \theta \rangle \neq 0 $. Then, for $n_k$ large enough, we have that  $\langle \sigma_*(d^{-\psi(n_k)} (D^{\psi(n_k)})_*( S^i_{n_k} )), \theta \rangle \neq 0$ for $i=1,2$. Since $\sigma_*(d^{-\psi(n_k)} (D^{\psi(n_k)})_*( S^i_{n_k} ))$ has support in $f^{{\psi(n_k)}}(\Delta^i_{{\psi(n_k)}})$, this contradicts the fact that   
$d\left( f^{\psi(n_k)}(\Delta^1_{n_k}),f^{\psi(n_k)}(\Delta^2_{n_k})  \right) \geq \delta $. Hence $h^1_\top(f) =0$.

\end{document}